\documentclass[12pt]{amsart}
\usepackage{amsfonts, amssymb, latexsym}

\newtheorem*{claim*}{Claim}
\newtheorem*{Lemma*}{Lemma}
\newtheorem*{problem*}{Problem}

\newtheorem{Main Conjecture}[theorem]{Main Conjecture}

\theoremstyle{remark}

\theoremstyle{plain}

\newtheorem*{MainThm}{Theorem}

\newtheorem*{corollaryA}{Corollary A}
\newtheorem*{corollaryB}{Corollary B}
\newtheorem*{corollaryC}{Corollary C}

\newcommand{\excise}[1]{}

\begin{document}
\pagestyle{plain}

\title{symmetric group representations and ${\mathbb Z}$}

\author{Anshul Adve}
\address{University Laboratory High School, Urbana, IL 61801}
\email{anshul.adve@gmail.com}
\author{Alexander Yong}
\address{Dept. of Mathematics, University of Illinois at
Urbana-Champaign, Urbana, IL 61801}
\email{ayong@uiuc.edu}


\date{June 30, 2017}
\maketitle

Let ${\mathfrak S}_n$ be the symmetric group of permutations of $\{1,2,\ldots,n\}$. A
\emph{representation}
is a homomorphism
$\rho:{\mathfrak S}_n\to {\sf GL}(V)$ where $V$ is a vector space
over ${\mathbb C}$. Equivalently, $V$ is an ${\mathfrak S}_n$-module under the action defined by
$\sigma\cdot v=\rho(\sigma)v$, for $\sigma\in {\mathfrak S}_n$ and $v\in V$. Then $\rho$
is \emph{irreducible} if there is no proper ${\mathfrak S}_n$-submodule of $V$. Conjugacy classes and hence irreducible
representations
of ${\mathfrak S}_n$ biject with ${\sf Par}(n)$, the partitions of size $n$.

Consider
three families of numbers from the theory:

\begin{itemize}
\item[(I)] The \emph{character} of $\rho$ is
\[\chi^{\rho}: \ {\mathfrak S}_n\to {\mathbb C}; \ \ \ \ \  \sigma   \mapsto {\rm tr}(\rho(\sigma)).\]
Textbooks focus on the case $V=V_{\lambda}$ is irreducible (because of Maschke's theorem).
Since characters are constant on
each conjugacy class $\mu$, one needs only $\chi^{\lambda}(\mu)$. These are
computed by the
Murnaghan-Nakayama rule (see below). More recent results include bounds on
(normalized) character evaluations \cite{Roichman, Feray}.
\item[(II)] If $V_{\lambda}$ and $V_{\mu}$ are irreducible ${\mathfrak S}_m$ and
${\mathfrak S}_n$-modules, respectively, then $V_{\lambda}\otimes V_{\mu}$ is an
irreducible ${\mathfrak S}_m\times {\mathfrak S}_n$-module. If $V_{\nu}$ is an irreducible
${\mathfrak S}_{m+n}$-representation,
    it restricts to a ${\mathfrak S}_m\times {\mathfrak S}_n$-representation $V_{\nu}\!\downarrow_{{\mathfrak S}_m\times
    {\mathfrak S}_n}^{{\mathfrak S}_{m+n}}$. The \emph{Littlewood-Richardson coefficient} is
\[c_{\lambda,\mu}^{\nu}=\text{multiplicity of $V_{\lambda}\otimes V_{\mu}$ in $V_{\nu}\!\downarrow_{{\mathfrak S}_m\times
    {\mathfrak S}_n}^{{\mathfrak S}_{m+n}}$}.\]
Many \emph{Littlewood-Richardson rules} are available to count
$c_{\lambda,\mu}^{\nu}$ \cite{ECII}.
\item[(III)] If $V_{\lambda}, V_{\mu}$ are ${\mathfrak S}_n$-modules then so is $V_{\lambda}\otimes V_{\mu}$. Hence we may write
\[V_{\lambda}\otimes V_{\mu}\cong \bigoplus_{\nu\in {\sf Par}(n)} V_{\nu}^{\oplus g_{\lambda,\mu,\nu}}.\]
 Here, $g_{\lambda,\mu,\nu}$ is the \emph{Kronecker coefficient}. One has an ${\mathfrak
S}_3$-symmetric but cancellative formula
$g_{\lambda,\mu,\nu}=\frac{1}{n!}\sum_{\sigma\in {\mathfrak S}_n}
\chi^{\lambda}(\sigma)\chi^{\mu}(\sigma)\chi^{\nu}(\sigma)$;
it is an old open problem to give a manifestly nonnegative combinatorial rule. The study of Kronecker coefficients has been
given new impetus from \emph{Geometric Complexity Theory}, an approach to
the ${\sf P}$ vs ${\sf NP}$ problem; see \cite{Jonah} and the references therein.
\end{itemize}

This note visits a rudimentary point.
While for finite groups, character
evaluations are 
algebraic integers, for ${\mathfrak S}_n$, in fact
$\chi^{\lambda}(\mu)\in {\mathbb Z}$. Moreover,
	by definition, $c_{\lambda,\mu}^{\nu},
g_{\lambda,\mu,\nu}\in {\mathbb Z}_{\geq 0}$.
We remark the three
converses hold.\footnote{inspired by
P.~Polo \cite{Polo}: every $f\in 1+q{\mathbb Z}_{\geq 0}[q]$ is a
Kazhdan-Lusztig
polynomial for some ${\mathfrak S}_n$} The proof uses standard facts,
but we are unaware of any specific reference
in the textbooks \cite{James, Fulton.Harris, ECII, Sagan}, or elsewhere.

\begin{MainThm}
Every integer is infinitely often an irreducible ${\mathfrak S}_n$-character evaluation. Every
nonnegative
integer is infinitely often a Littlewood-Richardson coefficient, and a Kronecker coefficient.
\end{MainThm}

\begin{corollaryA}
There exists a value-preserving multiset bijection between the
Littlewood-Richardson and Kronecker coefficients.
\end{corollaryA}
\begin{proof}
	Clearly, the Theorem implies that for each $k\in {\mathbb Z}_{\geq 0}$, the sets
\[{\sf LR}_k=\{(\lambda,\mu,\nu):c_{\lambda,\mu}^{\nu}=k\} \text{\ and \ } {\sf Kron}_k=\{(\lambda,\mu,\nu):g_{\lambda,\mu,\nu}=k\}\]
are countably infinite and thus
in bijection.
\end{proof}

Desirable would be a construction of an
injection ${\sf Kron}_k\hookrightarrow {\sf LR}_k$ for each $k\in {\mathbb Z}_{\geq 0}$ (avoiding the
countable axiom of choice). That should solve
the Kronecker problem in (III), by reduction
to (II). This we cannot do. However,
there has been success in this vein \cite{KMS} on another counting problem. See the Remark
at the end of this paper.\footnote{There is debate about the idiomatic meaning of \emph{counting rule} or
\emph{manifestly nonnegative combinatorial rule} etc.
Consider the
(adjusted) Fibonacci numbers ($1,1,2,3,5,8,13,\ldots$). A counting rule is that $F_n$
counts the number of $(1,2)$-lists
whose
sum is $n$. The recursive (and computationally efficient) description is
$F_n=F_{n-1}+F_{n-2}$ ($n\geq 2$) where $F_0=F_1=1$. Construct a
binary tree ${\mathcal T}_n$ with root labelled $F_n$; each node of label $F_i$ has a
left child
$F_{i-1}$ and right child $F_{i-2}$. Leaves of ${\mathcal T}_n$ are labelled $F_1$ or
$F_0$. $F_n$
counts the number of leaves of ${\mathcal T}_n$. The latter
description restates the recurrence and is not, \emph{per se}, a counting rule.}

\noindent
\emph{Proof of the Theorem:}
The
\emph{Murnaghan-Nakayama rule}
states
$\chi^{\lambda}(\mu)=\sum_{T} (-1)^{{\rm ht}(T)}$,
where
$T$ is a tableaux of shape $\lambda$ with $\mu_i$ many labels $i$, the entries
are weakly increasing along rows and columns, and the labels $i$ form a connected skew shape $T_i$
with no $2\times 2$ subsquare;
${\rm{ht}}(T)$ is the sum of the heights of each $T_i$, i.e., one less than the number of rows of $T_i$.

We sharpen the assertion about $\chi^{\lambda}(\mu)$.
In particular, for a given $n$, we consider the intervals of consecutive integers achievable as character evaluations for $\mathfrak{S}_n$.
From the rule, the character of the \emph{defining
representation}
satisfies
$\chi^{(n-1,1)}(\mu)=\#(\text{$1$'s in $\mu$})-1$
(see also \cite[Lemma~6.9]{James}).
Hence, $\chi^{(n-1,1)}$ takes the values $[0,n-2]$.
Similarly, $\chi^{(2,1^{n-2})}$ achieves an interval of negative integers:
Take $k \in [1,n-5] \cup \{n-3\}$.
If $k\not\equiv n \text{ \ mod $2$}$, let $\mu = (n-k-1,1^{k+1})$.
Otherwise,
if $k\equiv n \text{\ mod $2$}$, let $\mu = (n-k-4,3,1^{k+1})$.
Note that if $k=n-6$, let $\mu$ be these parts in  decreasing order.
In either case, the rule shows
$\chi^{(2,1^{n-2})}(\mu) = -k$.
Thus, for $n \ge 5$, $[-(n-5),n-2]\subseteq\{\chi^\lambda(\mu) :
\lambda,\mu \in {\sf Par}(n)\}$.
Taking $n\to \infty$ implies the statement regarding character evaluations.

The \emph{Kostka coefficient} $K_{\lambda,\mu}$ is the number of
\emph{semistandard Young tableaux} of shape $\lambda$ with content $\mu$, i.e., fillings of $\lambda$ with $\mu_i$ many $i$'s such that rows
are weakly increasing and columns are strictly increasing.
\begin{Lemma*}
Every nonnegative integer is infinitely often a Kostka coefficient.
\end{Lemma*}
\begin{proof}
Clearly,
$K_{(1+j,1^{k-1}),(j,1^k)}=k \text{\ \  for $j\geq 1$.}$
The lemma then follows.
\end{proof}

The Littlewood-Richardson coefficient
claim holds since it is long known that Kostka coefficients are a special case.
To be specific,
$K_{\lambda,\mu}=c_{\sigma,\lambda}^{\tau}$
where
\[\tau_i=\mu_i+\mu_{i+1}+\cdots, \ i=1,2,\ldots,\ell(\mu), \text{\ and }\]
\[\sigma_i=\mu_{i+1}+\mu_{i+2}+\cdots, \  i=1,2,\ldots,\ell(\mu)-1.\footnote{This reduction is used
by
H.~Narayanan \cite{Nara} to show computing $c_{\lambda,\mu}^{\nu}$ is a $\#{\sf P}$ problem.}\]

For $\lambda=(\lambda_1,\lambda_2,\ldots)$, let
$\lambda[N]:=(N-|\lambda|,\lambda_1,\lambda_2,\ldots)$. F.~D.~Murnaghan \cite{Murn} proved that
for an
integer $N\gg 0$,
$\chi^{\lambda[N]}\otimes \chi^{\mu[N]}=\sum_{\nu}
\overline{g_{\lambda,\mu,\nu}}
\chi^{\nu[N]}$.
The $\overline{g_{\lambda,\mu,\nu}}$ are called \emph{stable Kronecker
coefficients} and are evidently a special case of Kronecker coefficients. When
$|\lambda|+|\mu|=|\nu|$ one has
${\overline{g_{\lambda,\mu,\nu}}}=c_{\lambda,\mu}^{\nu}$.
Hence one infers the
Kronecker coefficient assertion.
\qed

When, e.g., $n=25$, all of
$[-853,949]$ appear as some $\chi^{\lambda}(\mu)$, but the proof merely guarantees $[-20,23]$.
Let $\ell_n$ be the maximum size of an interval of consecutive character evaluations for ${\mathfrak S}_n$. Trivially,
the results of \cite{Roichman, Feray} imply upper bounds for $\ell_n$. Can
one
prove better upper or lower bounds for $\ell_n$?

Let ${\sf A}_n$ be the \emph{alternating group} of even permutations in ${\mathfrak S}_n$.
Sources about the representation theory of ${\sf A}_n$ include
\cite[Section~2.5]{JamesKerber} and \cite[Section~5.1]{Fulton.Harris}.
Character evaluations of ${\sf A}_n$ are not always integral, however:
\begin{corollaryB}
Every integer appears infinitely often as an ${\sf A}_n$-irreducible character evaluation.
\end{corollaryB}
\begin{proof}
Let
$\psi^{\lambda}=\chi^{\lambda}\!\downarrow^{{\mathfrak S}_n}_{{\sf A}_n}$
be the
character of the restriction
of the ${\mathfrak S}_n$-irreducible $V_{\lambda}$. If $\mu$ is not a partition with distinct odd parts
then the conjugacy class in ${\mathfrak S}_n$ of cycle type $\mu$ is also a conjugacy class of ${\sf A}_n$. If $\lambda$ is not a self-conjugate partition, the restriction is an
${\sf A}_n$-irreducible and also
$\psi^{\lambda}(\mu)=\chi^{\lambda}(\mu)$.
Repeat the Theorem's character
argument, since for $n\geq 4$ neither $\lambda$ used is self-conjugate, and since
for $k\geq1$, $\mu$ has equal parts.
\end{proof}
\excise{
We tabulated the largest consecutive interval that appears among the character evaluations for $n=1,2,\ldots,25$:

\begin{table}[h]
\centering
\begin{tabular}{|c|c|c|c|c|c|c|c|c|c|}
\hline
$n$ & $1$ & $2$ & $3$ & $4$ & $5$ & $6$ & $7$ & $8$ & $9$\\
\hline
\text{Interval} & $[-1,2]$ & $[-2,2]$ & $[-2,3]$ & $[-2,4]$ & $[-3,3]$ & $[-4,4]$ & $[-7,7]$ &
$[-7,11]$ & $[-7,12]$
\\
\hline
\end{tabular}
\end{table}

\begin{table}[h]
\ \ \ \ \ \ \ \ \ \ \ \ \ \ \ \  \ \ \ \ \ \ \ \ \ \ \ \ \ \ \ \ \ \ \ \ \
\begin{tabular}{|c|c|c|c|c|c|c|c|c|}
\hline
$10$ & $11$ & $12$ & $13$ & $14$ & $15$ & $16$ & $17$ & $18$\\
\hline
$[-11,12]$ & $[-7,19]$ & $[-17,17]$ & $[-19,23]$ & $[-31,31]$ & $[-31,47]$ & $[-53,59]$ & $[-59,73]$ & $[-101,123]$
\\
\hline
\end{tabular}
\end{table}

\begin{table}[h]
\ \ \ \ \ \ \ \ \ \ \ \ \ \ \ \  \ \ \ \ \ \ \ \ \ \ \ \ \ \ \ \ \ \ \ \ \
\begin{tabular}{|c|c|c|c|c|c|c|}
\hline
$19$ & $20$ & $21$ & $22$ & $23$ & $24$ & $25$ \\
\hline
$[-137,193]$ & $[-157,157]$ & $[-271,263]$ & $[-403,403]$ & $[-397,443]$ & $[-521,557]$ & $[-853,949]$\\
\hline
\end{tabular}
\end{table}
The largest interval guaranteed by our argument is $[-(n-5),n-2]$, which is clearly small compared to the data. Can one prove larger intervals or in any case better understand the values that can occur for given $n$.}

\noindent
{\bf Definition.} For a countable
indexing set ${\sf A}$,
a family of nonnegative integers
$(a_{\alpha})_{\alpha\in {\sf A}}$ is \emph{entire} if
every $k\in
{\mathbb Z}_{\geq 0}$ appears infinitely often.

\medskip
Many of the nonnegative integers arising in
algebraic combinatorics are entire. For example, this is true for the theory of \emph{Schubert polynomials}
(we refer to \cite{Manivel} for references).
If $w_0\in {\mathfrak S}_n$ is the longest permutation
then ${\mathbb S}_{w_0}(x_1,\ldots,x_n)=x_1^{n-1}x_2^{n-2}\cdots x_{n-1}$. If $w\neq w_0$, $w(i)<w(i+1)$ for some $i$. Then
${\mathbb S}_w(x_1,\ldots x_n)=\partial_i {\mathbb S}_{ws_i}(x_1,\ldots,x_n)$
where $\partial_i=\frac{f-s_i(f)}{x_i-x_{i+1}}$ and $s_i$ is the simple transposition
interchanging $i,i+1$. Nontrivially, each ${\mathbb S}_w\in {\mathbb Z}_{\geq 0}[x_1,x_2,\ldots]$.
Moreover, ${\mathbb S}_w={\mathbb S}_{w\times 1}$ where $w\times 1\in {\mathfrak S}_{n+1}$ is the usual image of $w\in \mathfrak{S}_n$.
Thus we can discuss ${\mathbb S}_w$ for $w\in {\mathfrak S}_{\infty}$; these form a ${\mathbb Z}$-linear basis of
${\mathbb Z}[x_1,x_2,\ldots]$. The \emph{Schubert structure constants}
$C_{u,v}^w:=[{\mathbb S}_w]{\mathbb S}_u {\mathbb S}_v\in {\mathbb Z}_{\geq 0}$
for geometric reasons. The \emph{Stanley symmetric function} is defined by
$F_w=\lim_{m\to\infty} {\mathbb S}_{1^m\times w} \in {\mathbb Z}[[x_1,x_2,\ldots]]$;
here $1^m\times w\in {\mathfrak S}_{m+n}$ sets $1^m\times w(i)$ equal to $i$ if $1\leq i\leq m$ and equal to $w(i-m+1)+m$
otherwise.  $F_w$ is Schur-nonnegative.

\begin{corollaryC}
These families of nonnegative integers are entire:
\begin{itemize}
\item[(a)] The coefficients of monomials in Schubert polynomials.
\item[(b)] The Schubert structure constants.
\item[(c)] The coefficients of Schur functions in Stanley symmetric functions.
\end{itemize}
\end{corollaryC}
\begin{proof} (a) is true by the Lemma since when $w$ is \emph{Grassmannian} 
(has at most one descent),
${\mathbb S}_w(x_1,\ldots,x_n)$ is a
Schur polynomial $s_{\lambda}$.
When $u,v$ and $w$ are Grassmannian with descent position $d$, then $C_{u,v}^w$ is a Littlewood-Richardson coefficient so the
Theorem implies (b).
Finally, when $w$ is $321$-avoiding (i.e., there does not exist indices $i<j<k$ such that $w(i)>w(j)>w(k)$),
$F_w=s_{\nu/\lambda}=\sum_{\mu}c_{\lambda,\mu}^{\nu} s_{\mu}$
is a skew Schur
function. Hence, here the coefficient (c) is $c_{\lambda,\mu}^{\nu}$ and we apply the Theorem.
\end{proof}

Abstractly, all entire families are mutually in
value-preserving bijection. However, for Corollary~C one
can say more: (a) and (c) are a special cases of (b) (see \cite{Bergeron.Sottile} and \cite{Yong:MRL}). Can one construct a ``wrong way map'' (as in ${\mathbb Q}\hookrightarrow {\mathbb N}$)
for either (b)$\hookrightarrow$(a) or (b)$\hookrightarrow$(c)
(thereby finding a rule for $C_{u,v}^w$)?  A special case indicating the difficulty is:

\begin{problem*}
Construct an explicit value-preserving injection between
Littlewood-Richardson and Kostka coefficients.
\end{problem*}

\noindent
{\bf Remark.} Finding a wrong way map has solved a significant counting
rule problem concerning
A.~Buch-W.~Fulton's \emph{quiver coefficients}. These arise in the study of degeneracy
loci of vector bundles over a smooth projective algebraic variety. It was conjectured by
those two authors that these integers are nonnegative, with a conjectural counting
rule. Also, A.~Buch showed that special cases of the quiver
coefficients are
the numbers from (c) above. The
resolution of this problem,
due to A.~Knutson-E.~Miller-M.~Shimozono, came by establishing the \emph{opposite}:
quiver coefficients are special cases of the well-understood numbers (c).
We refer to the solution \cite{KMS} for background and references.

 \subsection*{Acknowledgements} We thank Ezra Miller and Bernd Sturmfels for
stimulating remarks (over a dozen years ago) concerning \cite{KMS, Yong:MRL}. We
are grateful to Alexander Miller suggesting the character evaluation
argument, for data, as well as many other helpful comments.
We thank
Nantel Bergeron, Colleen Robichaux, Hugh Thomas and Mike Zabrocki for informative
discussions. AY
was partially
supported by an NSF grant.

\end{document}